\RequirePackage{fix-cm}
\documentclass[smallextended]{svjour3}       
\smartqed  
\usepackage{enumerate}
\usepackage{graphicx}
\usepackage{mathtools}
\usepackage{amsfonts,amssymb,amsmath}
\usepackage{hyperref}
\usepackage{color}

\newcommand{\R}{{\mathbb{R}}}

\newcommand{\N}{{\mathbb{N}}}

\begin{document}

\title{ On the feasibility and convergence of the inexact Newton method under minor conditions on the error terms}
\subtitle{}


\author{Eduardo Ramos \and
Marcio Gameiro \and Victor Nolasco
}


\institute{E. Ramos \at
Instituto de Ci\^{e}ncias Matem\'{a}ticas e de Computa\c{c}\~{a}o, Universidade de S\~{a}o Paulo, Caixa Postal 668, 13560-970, S\~{a}o Carlos, SP, Brazil \\
\email{eduardoramos@usp.br}
\and
M. Gameiro \at
Instituto de Ci\^{e}ncias Matem\'{a}ticas e de Computa\c{c}\~{a}o, Universidade de S\~{a}o Paulo, Caixa Postal 668, 13560-970, S\~{a}o Carlos, SP, Brazil \\
\email{gameiro@icmc.usp.br}
\and
V. Nolasco \at
Instituto de Ci\^{e}ncias Matem\'{a}ticas e de Computa\c{c}\~{a}o, Universidade de S\~{a}o Paulo, Caixa Postal 668, 13560-970, S\~{a}o Carlos, SP, Brazil \\
\email{victor.nolasco@usp.br}
}

\date{Received: date / Accepted: date}

\maketitle

\begin{abstract}
In this paper we introduce a semi-local theorem for the feasibility and convergence of the inexact Newton method, regarding the sequence $x_{k+1} = x_k - Df(x_k)^{-1}f(x_k) + r_k$, where $r_k$ represents the error in each step. Unlike the previous results of this type in the literature, we prove the feasibility of the inexact Newton method under the minor hypothesis that the error $r_k$ is bounded by a small constant to be computed, and moreover we prove results concerning the convergence of the sequence $x_k$ to the solution under this hypothesis. Moreover, we show how to apply this this method to compute rigorously zeros for two-point boundary value problems of Neumann type. Finally, we apply it to a version of the Cahn–Hilliard equation.

\end{abstract}

\section{Introduction}

Regarding semi-local guarantees of convergence for the Newton method we have the celebrated Kantorovich (or Newton-Kantorovich) theorem, which gives sufficient conditions for the convergence of the Newton method $x_{k+1}=x_k-Df(x_k)^{-1} f(x_k)$ to a zero of $f$. On the other hand, it should be noted that the exact Newton method is usually impracticable for the purpose of rigorous computations since we cannot expect in general to perform the computations of $x_k-Df(x_k)^{-1}f(x_k)$ without numerical errors. These numerical errors can be incorporated in the computations by means of error bounds, but in general can not be avoided, specially in infinite dimensional setting. One option to circumvent this problem is to consider the more general inexact Newton method $x_{k+1}=x_k-Df(x_k)^{-1}f(x_k)+r_k$ in which $r_k$ represents the numerical error during each step of the computation.

In this regard, an important question is to give semi-local conditions, that is, conditions in terms on the initial point $x_0$ analogous to the ones in the Newton-Kantorovich theorem in order for the inexact Newton method to be feasible and converge under conditions on the decaying rate of $r_k$.

In the recent literature, such as \cite{2019-Argyros,2014-Argyros,2012-Ferreira,2008-Mi,2009-Weipin}, semi-local conditions for the inexact Newton method to be feasible and convergent have been given in the form $\left\|r_k\right\|\leq \eta\left\|f(x_k)\right\|$ under conditions on the bounding term $\eta$. On the other hand, since the sequence $x_k$ is know to converge fast to the zeros of $f$, it follows that $f(x_k)$ converges fast to $0$ and thus the same must hold for error term $r_k$ under these conditions. This fast decay rate for $r_k$ on the other hand is a very strong condition, which is hard to obtain in actual applications.

With this in mind, we present in this paper a condition for the feasibility under the minor condition that the error terms $r_k$ be bounded by a small constant $d$ to be computed, and we prove convergence of the sequence $x_{k+1}=x_k-Df(x_k)^{-1}f(x_k) + r_k$ under the minor condition that $r_k\to 0$ as $k\to \infty$. To our knowledge, this is the first time that such kind of conditions has been given to the inexact Newton method. Moreover, our condition is very similar to that of the Newton-Kantorovich theorem, in that the crucial condition for our result to hold is the existence of zeros of a related unidimensional function $g(t)$ given by
\begin{equation*}
g(t)=\eta_d-1_dt+\frac{3K}{2}t^2
\end{equation*}
where $\eta_d=\eta+d$, $1_d=1+Kd$, $\eta$ is a bound for $Df(x_0)^{-1}f(x_0)$ and $K$ a Lipschitz constant for the function $Df(x_0)^{-1}Df(x)$ in the domain. Thus, together with bijectivity modulus \cite{RamosTwopoint2020}, we show how to compute rigorously zeros for a differential operator. Finally, we apply it in a version of the Cahn–Hilliard equation. As we shall show our approach has a simple form even when it is applied in non-linear problems, against the most part of non-linear methods. 


The paper is organized as follows: In Section~\ref{sec:Preliminaries} we present the preliminaries results used to prove the main result of the paper. In Section~\ref{sec:InexactNewtonMethod} we present the main result, which corresponds to a new theorem for the feasibility and convergence of the inexact Newton method. In Section~\ref{sec:BijectivityModulus} show a variation of the main theorem more suitable for applications, using a new definition called bijectivity modulus of an operator. In Section~\ref{sec:Applications} we apply the previous result in a version of the Cahn–Hilliard equation. In Section~\ref{sec:Conclusion} the conclusion is presented.

\section{Preliminaries}
\label{sec:Preliminaries}

Let $(X,\left\|.\right\|_X)$ and $(Y,\left\|.\right\|_Y)$ denote Banach spaces and $(\mathcal{L}(X,Y),\left\|.\right\|_{\mathcal{L}(X,Y)})$ denote the Banach space of bounded linear operators from $X$ to $Y$ under the strong norm $\left\|.\right\|_{\mathcal{L}(X,Y)}$ given by $\left\|G\right\|_{\mathcal{L}(X,Y)} = \sup_{\left\|x\right\|_X\leq 1}\left\|G(x)\right\|_Y$ for all $G\in \mathcal{L}(X,Y)$. Moreover, we denote by $B(x_0,r)_{X}$ and $\bar{B}(x_0,r)_{X}$, respectively, the open and closed balls with radius $r$ centered at $x_0$ under the $\left\|.\right\|_{X}$ norm.

The following well known result (for a proof see \cite[Theorem 1.4, p. 192]{1980-Taylor}) tells us that the open ball of radius $1$ centered at the identity in $\mathcal{L}(X,Y)$ is composed only of invertible linear operators.

\begin{proposition}[Neumann series]
\label{thm:Neumann_series}
Given $F\in \mathcal{L}(X,Y)$, suppose that the inequality $\left\|I-F\right\|<1$ holds. Then $F$ is invertible with $F^{-1}\in \mathcal{L}(Y,X)$ and moreover
\begin{equation*} \left\|F^{-1}\right\|\leq \frac{1}{1-\left\|I-F\right\|}.
\end{equation*}
\end{proposition}
The following result is a direct consequence of this Proposition \ref{thm:Neumann_series}.

\begin{corollary}
\label{cor1}
Let $F\in \mathcal{L}(X,Y)$ and $G\in \mathcal{L}(Y,X)$ such that $F$ is invertible with $F^{-1}\in \mathcal{L}(Y,X)$ and suppose that $\left\|I - FG\right\|< 1$. Then $G$ is invertible with $G^{-1}\in \mathcal{L}(X,Y)$ and moreover
\begin{equation*}
\begin{aligned}
\left\|G^{-1}F^{-1}\right\|\leq \frac{1}{1-\left\|I-FG\right\|}
\end{aligned}
\end{equation*}
\end{corollary}

\begin{proof}
Due to Proposition \ref{thm:Neumann_series} we know that $H=FG$ is invertible with $H^{-1}\in \mathcal{L}(Y,Y)$ and
\begin{equation}
\label{hineq}
\left\|H^{-1}\right\|\leq \frac{1}{1-\left\|I-FG\right\|}.
\end{equation}
On the other hand, from the associativity property of compositions of linear operators and $G=F^{-1}H$, we have
\begin{equation*}
\begin{aligned}
(H^{-1}F)G = H^{-1}(FG) = H^{-1}H=I ~~\text{and} \\
G(H^{-1}F) = (F^{-1}H)(H^{-1}F)= I.
\end{aligned}
\end{equation*}
Thus, it follows that $G$ is invertible with inverse $G^{-1}=H^{-1}F$, which implies that $G^{-1}F^{-1}=H^{-1}$. Therefore the norm estimate in the proof follows from (\ref{hineq}).
\end{proof}

As a classical consequence of the Hahn-Banach Theorem (see \cite[Theorem 3.2, p. 134]{1980-Taylor}) we have the next result.

\begin{proposition}
\label{hahn}
Given a non-null element $x\in X$, there exists  $\phi\in \mathcal{L}(X,\R)$ such that $\left\|\phi\right\|_{\mathcal{L}(X,\R)}=1$ and $\phi(x)=\left\|x\right\|_X$.
\end{proposition}

We say that $f \colon U \subset X\to Y$ is differentiable if it is Fr\'{e}chet differentiable at every $x \in U$ and we denote its Fr\'{e}chet derivative operator by $Df \colon U \subset X \to \mathcal{L}(X,Y)$ (see \cite{1971-Cartan,2013-Garling}).

The following theorem can be seen as a generalization of the Taylor Theorem for twice differentiable functions, suitable for functions $f$ with Lipschitz continuous derivatives. Hence, we call it simply Taylor Theorem.

\begin{theorem}[Taylor Theorem]
\label{taylor}
Let $f \colon U \subset X\to Y$ be differentiable in the open set $U$, let $A \subset U$ be convex, $x_0 \in A$, $\kappa\in \R$ and suppose that
\begin{equation*}
\left\|Df(x)-Df(x_0)\right\|_{\mathcal{L}(X,Y)} \leq
\kappa \left\|x-x_0\right\|_X
\end{equation*}
for all $x\in A$. Then, it follows that
\begin{equation*}
\left\|f(x) - f(x_0) - Df(x_0)(x-x_0)\right\|_Y \leq \frac{\kappa}{2}\left\|x-x_0\right\|_X^2
\end{equation*}
for all $x \in A$.
\end{theorem}

\begin{proof}
If $x=x_0$, the theorem is trivially true. Otherwise, let
\begin{equation*}
r = \frac{1}{\left\|x-x_0\right\|_X^2}\left(f(x)-f(x_0)-Df(x_0)(x-x_0)\right).
\end{equation*}
The theorem is equivalent to proving that $\left\|r\right\|_Y\leq\frac{\kappa}{2}$. If $r= 0$, then once again there is nothing to prove. Otherwise, by Proposition~\ref{hahn}, there exists $\phi\in \mathcal{L}(Y,\R)$ such that $\left\|\phi\right\|_{\mathcal{L}(Y,\R)}=1$ and $\phi(r)=\left\|r\right\|_Y$.

Now, since $A$ is convex, it follows that $x_0+t(x-x_0)\in A$ for all $t\in [0,1]$ and since $U$ is open, there exist an open set $I$ such that $[0,1]\subset I$ and $x_0+t(x-x_0)\in U$ for all $t\in I$.

Thus,  if we denote $h=x-x_0$, by the chain rule it follows that the mapping $ g \colon I \to \R$ defined by $g(t)=\phi(f(x_0+th))$ is differentiable with $g'(t) = \phi(Df(x_0+t h)(h))$ for all $t\in [0,1]$. Moreover, since $\left\|\phi\right\|_{\mathcal{L}(Y,\R)}=1$ and since $\phi$ is linear, it follows by the hypothesis that
\begin{equation*}
\begin{aligned}\left|g'(t) - g'(0)\right|= \left|\phi((Df(x_0+t h)-Df(x_0))(h))\right| \\
\end{aligned}
\end{equation*}
\begin{equation*}
\begin{aligned}
\leq\left\|\phi\right\|_{\mathcal{L}(Y,\R)}\left\|Df(x_0+t h)-Df(x_0)\right\|_{\mathcal{L}(X,Y)}\left\|h\right\|_X\leq \kappa\left\|th\right\|_X\left\|h\right\|_X= \kappa t\left\|h\right\|_X^2
\end{aligned}
\end{equation*}
for all $t\in [0,1]$ and thus, from the Fundamental Theorem of Calculus, we have
\begin{equation*}
\begin{aligned}
\left|g(1) - g(0) - g'(0)\right| = \left|\int_0^1 (g'(t) - g'(0))\right|dt \leq \\
\end{aligned}
\end{equation*}
\begin{equation*}
\begin{aligned}
\int_0^1\left|g'(t)-g'(0)\right|dt \leq \int_0^1 \kappa t\left\|h\right\|_X^2 dt=\frac{\kappa \left\|h\right\|_X^2}{2}.
\end{aligned}
\end{equation*}
Finally, since $\left|\phi(r)\right|=\phi(r)=\left\|r\right\|_Y$, $g(0)=\phi(f(x_0))$, $g(1)=\phi(f(x))$, $g'(0)=\phi(Df(x_0)(h))$, and due to the linearity of $\phi$ we have
\begin{equation*}
\begin{aligned}
\left\|r\right\|_Y=\left|\phi(r)\right|=\left|\frac{1}{\left\|h\right\|_X^2}\left(\phi(f(x)) - \phi(f(x_0)) - \phi(Df(x_0)(h))\right)\right|=\\
\end{aligned}
\end{equation*}
\begin{equation*}
\begin{aligned}
\frac{1}{\left\|h\right\|_X^2}\left|g(1) - g(0) - g'(0)\right|\leq  \frac{\kappa}{2},
\end{aligned}
\end{equation*}
which concludes the proof.
\end{proof}

\section{Inexact Newton method}\label{sec:InexactNewtonMethod}

The Newton method, defined by the classical recurrence relation $x_{k+1}=x_k-Df(x_k)^{-1} f(x_k)$ remains one of the main computational methods to approximate solutions of nonlinear equations. In this section we present a version of the Kantorovich (or Newton-Kantorovich) theorem aimed to the Newton method when errors are allowed during the computation of each step, the so called \textit{Inexact Newton method}.
We begin by defining the Newton operator and feasibility for the Newton method.

\begin{definition}[Newton operator]
\label{newtop}
Let $X$ and $Y$ be Banach spaces, $U \subset X$ open, and $f \colon U \subset X \to Y$ differentiable in $U$. We define the \emph{Newton operator} $T_f \colon X \to X$ by
\begin{equation*}
T_f(x) =
\begin{cases}
x - Df(x)^{-1}f(x) & \text{if}~ x\in U ~\text{and}~ Df(x) ~\text{is invertible} \\
0 & \text{otherwise}
\end{cases}
\end{equation*}
\end{definition}

Regarding the above definition, although usually the Newton operator $T_f$ is defined only for values $x \in U$ such that $Df(x)$ is invertible, it is convenient for the subsequent definitions and theorems to extend its definition to other values $x$.

\begin{definition}[Feasibility for Newton method]
Let $X$ and $Y$ be Banach spaces, $U \subset X$ open, $A\subset U$, $f \colon U \subset X \to Y$ differentiable in $U$, and let $\{ x_k \}_{k \in \N}$ be a sequence in $X$. We say that $\{ x_k \}_{k \in \N}$ is \emph{feasible} for the Newton method for $f$ in $A$, if $x_k \in A$ and $Df(x_k)$ is invertible for all $k \in \N$.
\end{definition}
Thus, from the above definitions, if $\{ x_k \}_{k \in \N}$ is feasible for the Newton method in $A\subset U$, and $\{ x_k \}_{k \in \N}$ satisfies the recurrence relation $x_{k+1} = T_f(x_k)$ for all $k \in \N$, then it follows that
\begin{equation*}
x_{k+1} = x_k - Df(x_k)^{-1}f(x_k),
\end{equation*}
for all $k\in \N$, which corresponds to the classical Newton method.

Regarding the feasibility and convergence of the classical Newton method, we have the celebrated Newton-Kantorovich theorem as the main semi-local result guaranteeing feasibility and convergence of the sequence $\{ x_k \}_{k \in \N}$ to a zero of $f$ (see \cite{2017-Fernandez,1982-Kantorovich,2016-Lecerf}). Based on the Newton-Kantorovich theorem we propose the following theorem regarding the feasibility and convergence of the Inexact Newton method.

\begin{theorem}
\label{fundtheo}
Let $X$ and $Y$ be Banach spaces, $U \subset X$ an open set, $f \colon U \subset X \to Y$ a differentiable function in $U$, $x_0 \in U$, $R > 0$ satisfying $\bar{B}(x_0, R) \subset U$, let $\{ r_k \}_{k \in \N}$ be a sequence in $X$, and let $\eta \geq 0$, $K>0$ and $d>0$. Suppose that
\begin{enumerate}[$(a)$] 
\itemsep1em
\item $Df(x_0)$ is invertible with $Df(x_0)^{-1} \in \mathcal{L}(X,Y)$.
\item $\left\| Df(x_0)^{-1}f(x_0) \right\| \leq \eta$ and $\left\| Df(x_0)^{-1}(Df(x)-Df(y)) \right\| \leq K \left\| x-y \right\|$ for all $x,y \in \bar{B}(x_0,R)$.
\item $\left\| r_k \right\| \leq d$ for all $k \in \N$.
\item $g_d(t) = \eta_d-1_d t+\dfrac{3K}{2}t^2$ has a smallest real zero $t_d^*\leq R$ and moreover $d\leq \dfrac{1}{K}$, where $\eta_d = \eta+d$ and $1_d=1+Kd$.
\end{enumerate}

Then $0<t^*_d\leq \frac{1}{K}\left(1 - \frac{1}{\sqrt{3}}\right)$ and moreover

\begin{enumerate}[(i)]
\itemsep1em 
\item The function $f$ has a unique zero $x^* \in \bar{B}\left(x_0,t_d^*\right)$.
\item The sequence $\{ x_k \}_{k\in \N}$ defined by $x_{k+1} = T_f(x_k)+r_k$ for all $k\in \N$ is feasible for the Newton method for $f$ in $\bar{B}(x_0,t_d^*)$.
\item $\left\|x^*-x_{k+1}\right\|\leq \dfrac{\sqrt{3}}{2}K\left\|x^*-x_k\right\|^2+\left\|r_k\right\|$ for all $k\in \N$.
\item $\lim_{k\to \infty}x_k=x^*$ in case $\lim_{k\to \infty}\left\|r_k\right\|=0$.
\end{enumerate}
\end{theorem}

\begin{proof} We shall first prove that $0<t_d^*\leq \frac{1}{K}\left(1-\frac{1}{\sqrt{3}}\right)$:
 Notice that, defining $h(t)=t^2 - 4t + (1- 6K\eta)$, since $g_d(t)$ has a real zero by hypothesis, by the quadratic formula it follows that
\begin{equation*}
\left(1_d\right)^2 - 6K\eta_d \geq 0\Rightarrow 
(1+Kd)^2 - 6K(\eta+d) \geq 0 \Rightarrow 
\end{equation*}
\begin{equation*}
(Kd)^2 - 4\left(Kd\right) + (1 - 6K\eta) \geq 0 \Rightarrow h(Kd)\geq 0.
\end{equation*}
On the other hand, once again by the quadratic formula we notice that, since $\eta\geq 0$ and $K>0$, it follows that $h(t)$ has zeros $t_1$ and $t_2$ given by
\begin{equation*}
\begin{aligned}
t_1=2 - \sqrt{3+6K\eta}\mbox{ and }
t_2 = 2 + \sqrt{3+6K\eta},
\end{aligned}
\end{equation*} 
with $h(t)<0$ for all $t\in (t_1,t_2)$. Thus, since we proved above that $h(Kd)\geq 0$, it follows that $Kd\notin (t_1,t_2)$, but since due to item $(d)$ we have $Kd\leq 1\leq t_2$, we conclude that $Kd\leq t_1 = 2-\sqrt{3+6K\eta}$.On the other hand, the quadratic formula tells us that the smallest zero $t^*_d$ of $g_d(t)$ is given by
\begin{equation*} 0< t^*_d=\frac{1+Kd - \sqrt{(1+Kd)^2-6K(\eta+d)}}{3K}\leq \frac{1+Kd}{3K}
\end{equation*}
and since we proved above that $Kd\leq 2-\sqrt{3+6K\eta}$, it follows that
\begin{equation}\label{Ktl}Kt^*_d \leq \frac{1+Kd}{3}\leq 1-\sqrt{\frac{1}{3}+\frac{2}{3}K\eta}\leq 1-\frac{1}{\sqrt{3}}.
\end{equation}
which is what we wanted to prove.

Proof of $(i)$:
Given any $z\in \bar{B}(x_0,t_d^*)$ such that $Df(z)$ is invertible, we define $L_z:U\to X$ by $L_z(x)=x-Df(z)^{-1}f(x)$ for all $x\in U$. We shall prove in the following that $L_{x_0}$ is a contraction from $\bar{B}(x_0,t^*_d)$ to $\bar{B}(x_0,t^*_d)$. Thus, due to the Banach fixed point theorem (see \cite[Theorem 2.1, p. 24]{1982-Chow}) it will follow that $L_{x_0}$ has a unique fixed point, which in turn proves item $(i)$ since it is clear that every zero of $f$ in $\bar{B}(x_0,t^*_d)$ is a fixed point of $L_{x_0}$ and vice-versa.

The proof that $L_{x_0}$ is a contraction on $\bar{B}(x_0,t_d^*)$. will follow from the three items bellow, which, although more general than what needed to prove item $(i)$, will actually be necessary to prove item $(ii)$.
\begin{itemize}
\item[(I)] $\left\|L_{x_0}(x)-L_{x_0}(y)\right\|\leq c\left\|x-y\right\|$ for all $x,y\in \bar{B}(x_0,t_d^*)$ for some $c<1$.
\item[(II)] $Df(z)$ is invertible for all $z\in \bar{B}(x_0,t_d^*)$ with $\left\|Df(z)^{-1}Df(x_0)\right\|\leq \frac{1}{1-Kt_d^*}$.
\item[(III)] $L_z\left(\bar{B}(x_0,t_d^*)\right)\subset \bar{B}(x_0,t_d^*-d)$ for all $z\in \bar{B}(x_0,t_d^*)$.
\end{itemize}
Proof of $(I)$: Notice that due to the chain rule $L_{x_0}$ is differentiable with $DL_{x_0}(w)=I - Df(x_0)^{-1}Df(w)=Df(x_0)^{-1}(Df(x_0)-Df(w))$ for all $w\in U$.
Thus, letting $x$, $y$ in $\bar{B}(x_0,t_d^*)$, by the Mean Value Theorem and $(b)$ we have
\begin{equation*}
\begin{aligned}
\left\|L_{x_0}(x)-L_{x_0}(y)\right\|\leq \sup_{w\in \bar{B}(x_0,t^*_d)}\left\|DL_{x_0}(w)\right\|\left\|x-y\right\|\\[1ex]
=\sup_{w\in \bar{B}(x_0,t^*_d)}\left\|Df(x_0)^{-1}(Df(w)-Df(x_0))\right\|\left\|x-y\right\|\\[1ex]
\leq K\sup_{w\in \bar{B}(x_0,t^*_d)}\left\|w-x_0\right\|\left\|x-y\right\|\leq Kt_d^*\left\|x-y\right\|.
\end{aligned}
\end{equation*}
But, from (\ref{Ktl}) we have that $ Kt_d^*<1$ and thus letting $c=Kt_d^*$, item $(I)$ is proved.

Proof of $(II)$: Letting $z\in \bar{B}(x_0,t_d^*)$, due to the Mean Value Theorem, the inequality in (\ref{Ktl}) and $(b)$ it follows that
\begin{equation*}
\begin{aligned}
\left\|1-Df(x_0)^{-1}Df(z)\right\|=\left\|Df(x_0)^{-1}(Df(z)-Df(x_0))\right\|\leq \\  K\left\|z-x_0\right\|\leq Kt^*_d < 1.
\end{aligned}
\end{equation*}
Therefore, since $Df(x_0)^{-1}$ is invertible with inverse $Df(x_0)\in \mathcal{L}(X,Y)$, due to Corollary \ref{cor1} it follows that $Df(z)$ is invertible with $Df(z)^{-1}\in\mathcal{L}(Y,X)$ and
\begin{equation*}
\left\|Df(z)^{-1}Df(x_0)\right\|\leq \frac{1}{1-Kt^*_d}.
\end{equation*}
which proves item $(II)$.

Proof of $(III)$: Let $z\in \bar{B}(x_0,t_d^*)$ be fixed. From item $(II)$ $L_z$ is well defined and from the chain rule it follows that $L_z:U\to X$ is differentiable with $DL_z(y)=I-Df(z)^{-1}Df(y)$ for all $y\in U$, and in particular \begin{equation}\label{geralnow}
DL_z(y)-DL_z(z)=Df(z)^{-1}(Df(y)-Df(z))
\end{equation}
for all $y\in U$. Therefore, from item $(II)$, (\ref{geralnow}) and $(b)$ it follows that
\begin{equation*}
\begin{aligned}
\left\|DL_z(y)-DL_z(z)\right\|= \left\|(Df(z)^{-1}Df(x_0))( Df(x_0)^{-1}(Df(y)-Df(z)))\right\|\\
\leq \left\|Df(z)^{-1}Df(x_0)\right\|\left\| Df(x_0)^{-1}(Df(y)-Df(z))\right\|\leq \frac{K}{1-Kt^*_d} \left\|y-z\right\|
\end{aligned}
\end{equation*}
for all $y\in \bar{B}(x_0,t_d^*)$. Thus, since $DL_z(z)=I-Df(z)^{-1} Df(z)=0$, from the Taylor Theorem (Theorem \ref{taylor}) we have
\begin{equation}\label{L1}\left\|L_z(x)-L_z(z)\right\|=\left\|L_z(x)-L_z(z)-DL_z(z)(x-z)\right\|\leq \frac{K\left\|x-z\right\|^2}{2(1-Kt^*_d)},
\end{equation}
for all $x\in \bar{B}(x_0,t_d^*)$. Moreover, since by definition $L_z(x_0)=x_0-Df(z)^{-1}f(x_0)$, it follows from item $(II)$ and $(b)$ that
\begin{equation}\label{L2}\left\|L_z(x_0)-x_0\right\| = \left\|(Df(z)^{-1}Df(x_0))(Df(x_0)^{-1}f(x_0))\right\|\leq \dfrac{\eta}{1-Kt^*_d}.
\end{equation}
On the other hand, since from (\ref{Ktl}) we have $Kt^*_d<1$, it follows that $1-Kt_d^*\neq 0$ and thus from $g_d(t^*_d)=0$ we have
\begin{equation}\label{eqq}
\begin{aligned}
\eta+ d-(1+Kd)t^*_d + \frac{3K}{2} (t^*_d)^2 = 0\Rightarrow \\
\eta-(1-Kt^*_d)(t^*_d-d) + \frac{K}{2}(t^*_d)^2= 0
\Rightarrow \\
\frac{\eta}{1-Kt^*_d}- t^*_d+ \frac{K}{2(1-Kt^*_d)}(t^*_d)^2 =-d\Rightarrow\\
\frac{1}{1-Kt^*_d}\left(\eta+\frac{K}{2}(t^*_d)^2\right)= t^*_d-d.
\end{aligned}
\end{equation}
Thus from (\ref{L1}), (\ref{L2}) and (\ref{eqq}) it follows that
\begin{equation*}
\begin{aligned}
\left\|x_0-L_z(z)\right\| \leq  \left\|x_0-L_z(x_0)\right\|+\left\|L_z(x_0) - L_z(z)\right\| \\
\leq \frac{1}{1-Kt^*_d}\left(\eta+\dfrac{K(t^*_d)^2}{2}\right)= t^*_d-d.
\end{aligned}
\end{equation*}
which proves that $L_z(\bar{B}(x_0,t_d^*))\subset \bar{B}\left(x_0,t_d^*-d\right)$ and thus we conclude the proof of item $(III)$, and  therefore of item $(i)$.

Proof of $(ii)$:
Since $L_z(z)-x_0=z-Df(z)^{-1}f(z)-x_0$ it follows that if $r\in X$ with $\left\|r\right\|\leq d$ then due to item $(III)$ above we have
\begin{equation*}\left\|z-Df(z)^{-1}f(z)+r-x_0\right\|\leq \left\|x_0-L_z(z)\right\|+d\leq  t^*_d.
\end{equation*}
Combining this with the results obtained in the proof of $(i)$ we can conclude that, given $z\in \bar{B}(x_0,t^*_d)$, it follows that $Df(z)$ is invertible and  $z-Df(z)^{-1}f(z)+r\in \bar{B}(x_0,t^*_d)$ for all $r\in X$ with $\left\|r\right\|\leq d$. Thus, using induction, we can conclude that the sequence  $\{ x_k \}_{k\in \N}$ defined by $x_{k+1}=T_f(x_k)+r_k$ for all $k\in \N$ is such that $Df(x_k)$ is invertible and $x_k\in \bar{B}(x_0,t^*_d)$ for all $k\in \N$, which proves $(ii)$.

Proof of $(iii)$: Let
 $L_{x_m}$ be defined by $L_{x_m}(x)=x - Df(x_m)^{-1}Df(x)$ as in the proof of item $(i)$. Then it follows that $x^*=x^* - Df(x_m)^{-1}f(x^*)$, and thus from inequality (\ref{L1}) we have
\begin{equation}\label{zuni}
\begin{aligned}
\left\|x^*-x_{m+1}\right\|=\left\|L_{x_m}(x^*)-L_{x_m}(x_{m})-r_m\right\|\leq \frac{K\left\|x^*-x_m\right\|^2}{2(1-Kt^*_d)}+\left\|r_m\right\|.
\end{aligned}
\end{equation}
Now, notice that due to (\ref{Ktl}) we have
\begin{equation}\label{uni}
Kt_d^*\leq 1-\frac{1}{\sqrt{3}}\Rightarrow \frac{1}{2(1 - Kt_d^*)}\leq \frac{\sqrt{3}}{2},
\end{equation}
Thus we conclude from (\ref{zuni}) and (\ref{uni}) that 
\begin{equation*}
\begin{aligned}
\left\|x^*-x_{m+1}\right\|\leq \frac{\sqrt{3}}{2}K\left\|x^*-x_m\right\|^2+\left\|r_m\right\|
\end{aligned}
\end{equation*}
which proves item $(iii)$.

Proof of $(iv)$: Given $m\in \N$, since we proved $Kt^*_{d}\leq 1-\dfrac{1}{\sqrt{3}}$ and since from $(ii)$ we have that $x_m\in \bar{B}(x_0,t_d^*)$ it follows that
\begin{equation*}\frac{\sqrt{3}}{2}K\left\|x^*-x_m\right\|\leq \frac{\sqrt{3}}{2}Kt^*_{d}\leq \frac{\sqrt{3}-1}{2}.
\end{equation*}
Thus, from item $(iii)$ it follows that, for all $m\in \N$
\begin{equation}\label{ineqnow}\left\|x^*-x_{m+1}\right\|\leq \frac{\sqrt{3}}{2}K\left\|x^*-x_{m}\right\|^2 + \left\|r_m\right\|\leq \left(\frac{\sqrt{3} - 1}{2}\right)\left\|x^*-x_{m}\right\|+\left\|r_m\right\|.
\end{equation}
Now, given $\epsilon>0$, by definition of $\liminf$ there exists $N\in \N$ such that $k\geq N$ implies $\left\|r_k\right\|\leq \epsilon$. Thus, considering the sequence $\{ h_k \}_{k\geq N}$ defined by recurrence via
\begin{equation*}
\begin{aligned}
h_N = \left\|x^*-x_N\right\|\mbox{ and}\\
h_{k+1}=h_k\left(\frac{\sqrt{3}-1 }{2}\right)+\epsilon\mbox{ for all }k\geq N
\end{aligned}
\end{equation*}
it follows from inequality (\ref{ineqnow}) that $\left\|x^*-x_k\right\|\leq h_k$ for all $k\geq N$. On the other hand, one can verify via induction that the recurrence above implies in the formula
\begin{equation*}h_{N+k} = h_N \left(\frac{\sqrt{3}-1}{2}\right)^k+\sum_{m=0}^{k-1}\epsilon \left(\frac{\sqrt{3}-1}{2}\right)^m
\end{equation*}
for all $k\in \N$ and thus
\begin{equation*}\lim_{k\to \infty} h_k = \sum_{m=0}^\infty \epsilon\left(\frac{\sqrt{3}-1}{2}\right)^m = \frac{\epsilon}{1 - \left(\frac{\sqrt{3}-1}{2}\right)} =\left(1+\frac{1}{\sqrt{3}}\right) \epsilon.
\end{equation*}
Therefore, since $\left\|x^*-x_k\right\|\leq h_k$ for all $k\geq N$, it follows that
\begin{equation*}\liminf_{k\to \infty}\left\|x^*-x_k\right\|\leq \left(1+\frac{1}{\sqrt{3}}\right)\epsilon
\end{equation*}
and since $\epsilon>0$ was arbitrary chosen, letting $\epsilon\to 0^+$ in the above inequality, the result follows.
\end{proof}

\begin{remark}
For actual applications of the theorem above, we actually just compute the first $m+1$ terms $\{ x_k \}_{0 \leq k\leq m}$ and first $m$ error terms $\{ r_k \}_{0 \leq k \leq {m-1}}$ of the inexact Newton method and to apply the theorem we suppose that $r_k=0$ and $x_{k+1}=T_f(x_k)$ for all $k \geq m$.
\end{remark}

\section{Bijectivity Modulus and reformulation of the Main Result}\label{sec:BijectivityModulus}

To apply the Theorem \ref{fundtheo} to rigorously verify the existence of a true zero for a partial differential equation near a numerical zero, we will use the bijectivity modulus defined below.

\begin{definition} Let $X$ and $Y$ be Banach spaces. For $F\in \mathcal{L}(X,Y)$ we define the \textit{bijectivity} modulus $\lambda(F)$ of $F$ by

\begin{equation*}\lambda(F)=\begin{cases} \|F^{-1}\|^{-1}\mbox{ if }F\mbox{ is invertible,}\\
0\mbox{ otherwise.}
\end{cases}
\end{equation*}
\end{definition}
\begin{remark}\label{computeBijMod} If $F:\R^m\to \R^m$ is a linear operator such that $F(u)=Au^T$ for all $u\in \R^m$, where $A\in M_m(\R)$ then a lower bound for $\lambda(F)$ can be computed by proving that $A$ (and thus $F$ as well) is invertible and then computing an upper bound for $\left\|F^{-1}\right\|_{\mathcal{L}(\R^m,\R^m)}=\left\|A^{-1}\right\|_2\leq \left\|A^{-1}\right\|_F$ where $\left\|.\right\|_F$ corresponds to the Frobenius norm
$\left\|B\right\|_F = \sqrt{\sum_{i=1}^m\sum_{j=1}^m b_{ij}^2 }$ for all $B=(b_{ij})\in M_n(\R)$.
\end{remark}

\begin{proposition}
\label{basic2}
Given $F\in \mathcal{L}(X,Y)$ and $G\in \mathcal{L}(X,Y)$ we have that $\left|\lambda(F)-\lambda(G)\right|\leq \left\|F-G\right\|$.
\end{proposition}

\begin{proof}
See \cite{RamosTwopoint2020}.
\end{proof}

\begin{theorem}
\label{fundtheobij}
Let $X$ and $Y$ be Banach spaces, $U \subset X$ an open set, $f \colon U \subset X \to Y$ a differentiable function in $U$, $x_0 \in U$, $R > 0$ satisfying $\bar{B}(x_0, R) \subset U$, let $\{ r_k \}_{k \in \N}$ be a sequence in $X$, and let $\eta \geq 0$, $\nu\geq 0$, $K>0$ and $d>0$. Suppose that
\begin{enumerate}[$(a)$]
\itemsep1em
\item $\left\|f(x_0)\right\|\leq \eta$, $\lambda(Df(x_0))\geq \nu$ and $\left\| Df(x)-Df(y) \right\| \leq K \left\| x-y \right\|$ for all $x,y \in \bar{B}(x_0,R)$.
\item $\left\| r_k \right\| \leq d$ for all $k \in \N$.
\item $g_d(t) = \eta_d-\nu_d t+\dfrac{3K}{2}t^2$ has a a smallest real zero $t_d^*\leq R$ and moreover $d\leq \dfrac{\nu}{K}$, where $\eta_d=\eta+\nu d$ and $\nu_d=\nu+Kd$.
\end{enumerate}

Then $\nu>0$, $0<t_d^*\leq \frac{\nu}{K}\left(1-\frac{1}{\sqrt{3}}\right)$ and moreover

\begin{enumerate}[$(i)$]
\itemsep1em 
\item The function $f$ has a unique zero $x^* \in \bar{B}\left(x_0,t_d^*\right)$.
\item The sequence $\{ x_k \}_{k\in \N}$ defined by $x_{k+1} = T_f(x_k)+r_k$ for all $k\in \N$ is feasible for the Newton method for $f$ in $\bar{B}(x_0,t_d^*)$.
\item $\left\|x^*-x_{k+1}\right\|\leq \dfrac{\sqrt{3}}{2}K\nu^{-1}\left\|x^*-x_k\right\|^2+\left\|r_k\right\|$ for all $k\in \N$.
\item$\lim_{k\to \infty}x_k=x^*$ in case $\lim_{k\to \infty}\left\|r_k\right\|=0$.
\item $\left\|y-T_f(x_k)\right\|\leq \frac{\sqrt{3}}{\nu}\left\|Df(x_k)(x_k-y)-f(x_k)\right\|$.for all $y\in U$ and $k\in \N$.
\end{enumerate}
\end{theorem}
\begin{proof} Proof of (i) to (iv): Notice that from hypothesis $0<d\leq \dfrac{\nu}{K}$ and therefore $\nu>0$. Thus it follows that $\lambda(Df(x_0))\geq \nu>0$, which by definition implies that $Df(x_0)$ is invertible with $\|Df(x_0)^{-1}\|=\lambda(Df(x_0))^{-1}\leq \nu^{-1}$. Therefore, letting $\eta^*=\frac{\eta}{\nu}$ and $K^*=\frac{K}{\nu}$ it follows that
\begin{equation*}
\begin{aligned}
\|Df(x_0)^{-1}f(x_0)\|\leq \|Df(x_0)^{-1}\| \; \| f(x_0)\|\leq \eta^*
\end{aligned}
\end{equation*}
and
\begin{equation*}
\begin{aligned}
\|Df(x_0)^{-1}(Df(x)-Df(y))\|  \leq \|Df(x_0)^{-1}\| \; \|(Df(x)-Df(y))\| \leq K^*\|x-y\|
\end{aligned}
\end{equation*}
for all $x,y\in A$. Thus, the hypothesis $(a)$ and $(b)$ of Theorem \ref{fundtheo} are satisfied. Therefore, letting $\eta^*_d = \eta^*+\nu^*d$ and $\eta^*_d=\eta^*+K^*d$ from item $(ii)$ it follows that the polynomial $g^{*}_d(t)=\eta^*_d - \nu_d^* t + \dfrac{3 K^*}{2}t^2=\nu^{-1}g_d(t)$ has a smallest zero $t^*_d\leq R$, and thus from Theorem \ref{fundtheo}, items $(i)$ to $(iv)$ follows.

Proof of $(v)$: Notice that, due to Proposition \ref{basic2} it follows that
\begin{equation*}
\begin{aligned}\left|\lambda(Df(x_k))-\lambda(Df(x_0))\right|\leq \left\|Df(x_k)-Df(x_0)\right\|\\
\end{aligned}
\end{equation*}
\begin{equation*}
\begin{aligned}
\leq K\left\|x_k-x_0\right\|\leq Kt_d^*\leq \nu\left(1-\frac{1}{\sqrt{3}}\right)\Rightarrow
\end{aligned}
\end{equation*}
\begin{equation*}
\begin{aligned}
\lambda(Df(x_k))\geq \lambda(Df(x_0))-\nu\left(1-\frac{1}{\sqrt{3}}\right)\geq \nu-\nu\left(1-\frac{1}{\sqrt{3}}\right)=\frac{\nu}{\sqrt{3}}>0,
\end{aligned}
\end{equation*}
and thus
\begin{equation*}
\begin{aligned}
\left\|y-T_f(x_k)\right\|=\left\|Df(x_k)^{-1}(Df(x_k)(x_k-y)-f(x_k))\right\|
\end{aligned}
\end{equation*}
\begin{equation*}
\begin{aligned}
\leq \frac{\sqrt{3}}{\nu}\left\|Df(x_k)(x_k-y)-f(x_k)\right\|.
\end{aligned}
\end{equation*}
Proving the theorem.
\end{proof}
\begin{remark}\label{remark3}
Supposing the hypothesis of Theorem \ref{fundtheobij} is true and letting $\{x_k\}_{k\leq n}$ be such that $x_{k+1}=T_f(x_k)+r_k$ and  $\left\|r_k\right\|\leq d$ for all $k< n$, once again we can let $r_k=0$ for all $k\geq n$ and use the conclusions of Theorem \ref{fundtheobij}. In special, $(v)$ can be used to estimate $\left\|y-T_f(x_n)\right\|$ for any $y\in U$, and thus it can be used to estimate $\left\|x_{k+1}-T_f(x_k)\right\|$ where $x_{n+1}$ is the candidate for new term in the Inexact Newton Method.
\end{remark}

\section{Applications}\label{sec:Applications}

To apply Theorem \ref{fundtheobij}  we shall show how to use the Inexact Newton Method for the non-linear two point boundary value problem of Neumann type of the form
\begin{equation}\label{princeq}u'' = f(x,u),\ u'(0)=u'(1)
\end{equation}
where $f:\R^2\to \R$ is a bi-dimensional real function.

Letting $I=(0,1)$ and $I'=[0,1]$, Since it is known that $H^2(I)$ can be regarded as $C^1(I')$ functions (see \cite[Theorem 8.2]{2011-Brezis}), we let $H^2_N(I)=\{u\in H^2(I)\ |\ u'(0)=u'(1)\}$. It follows that the zeros $u\in H^2(I)$ of the two point boundary value problem in (\ref{princeq}) are the zeros of the operator
 $\mathcal{F}: H^2_N(I) \to L^2(I)$  defined by
 \begin{equation}\label{operatorf}\mathcal{F}(u) = u'' - f(x,u)\ \forall u\in H^2_N(I).
 \end{equation}
 Following \cite{RamosTwopoint2020}, the functions $\pi_{\cos}: L^2(I)\to \ell^2(\N)$ and $h_{\cos}: H^2_N(I) \to \ell^2(\N)$ defined by
\begin{equation*}
\begin{aligned}
\pi_{\cos}(u)=(\widehat{u}_{\cos}(0),\widehat{u}_{\cos}(1),\cdots)\mbox{ and }
h_{\cos}(u) = (1,\omega(1)\widehat{u}_{\cos}(1),\omega(2)\widehat{u}_{\cos}(2),\cdots)
\end{aligned}
\end{equation*}
are isometric isomorphisms, where $\{\widehat{u}_{\cos}(k)\}_{k\geq 0}$ is the sequence of coefficients of $u$ in $L^{2}(I)$ cosine basis (see \cite[p. 145]{2011-Brezis}), and $\omega(k) = \sqrt{1+(\pi k)^2+(\pi k)^4}$ for $k\in \N$. Thus, we define $\pi_{\cos,m} :L^2(I)\to \R^m$ by
\begin{equation*}\pi_{\cos,m}(u)=(\widehat{u}_{\cos}(0),\cdots,\widehat{u}_{\cos}(m-1))
\end{equation*}
and let $h^{-1}_{\cos,m}:\R^m\to H^2_N(I)$ be the restriction of $h^{-1}_{\cos}$ to $\R^m\subset \ell^2(\N)$. Given an operator $\mathcal{G}:H^2_N(I)\to L^2(I)$, it is reasonable to consider the finite dimensional operator $\mathcal{G}_{\cos,m}:\R^m\to \R^m$ defined by
\begin{equation*}\mathcal{G}_{cos,m}=\pi_{\cos,m} \circ \mathcal{G} \circ h^{-1}_{\cos,m}
\end{equation*}
as a natural finite dimensional approximation for $\mathcal{G}$. Moreover, we define $\left\|g\right\|_{C^0(I')}=\sup_{x\in I'}\left|g(x)\right|$ for all $g\in C^0(I')$, $\left\|g\right\|_{C^1(I')}=\sup_{x\in I'}\sqrt{g(x)^2+g'(x)^2}$ for all $g\in C^1(I')$, and $c_1=\left(\tanh(1)\right)^{-\frac{1}{2}}=\sqrt{\frac{\exp(2)+1}{\exp(2)-1}}$.

With these definitions in mind, given $m>0$, and $f:\R^{2}\to \R$ a $C^2$ elementary function, and an initial point $u_0=h^{-1}_{\cos,m}(b_0)\in H^2_N(I)$, where $b_0\in \R^m$, we show in the following how to fulfill all items needed to apply the Inexact Newton Method and Theoerem \ref{fundtheobij} for $\mathcal{F}$ beginning at any point $u_0$ such that $u_0=h_{\cos}^{-1}(b_0)$, for $b_0\in \R^m$.
\vspace{0.3cm}

\noindent \textbf{Computation of $\eta$:}
\vspace{0.3cm}

Since $u_0$ is an elementary function (finite sum of cosine functions) it follows that $\left\|\mathcal{F}(u_0)\right\|$  will correspond to the integral of an elementary function, and thus we can compute a fine interval enclosure for these value using the Simpson rule with explicit error term (see Theorem 12.1 in \cite{2010-Rump}).

\vspace{0.3cm}
\noindent \textbf{Computation of $\nu$:}
\vspace{0.3cm}

Following \cite{RamosTwopoint2020} and supposing that
\begin{equation*}
\begin{aligned}
\left\|f_u(x,u_0)\right\|_{C^0(I)}+\left\|f_u(x,u_0)\right\|_{C^1(I)}\leq N,
\end{aligned}
\end{equation*}
we have $\mathcal{F}:H^2_N(I)\to L^2(I)$, as defined in (\ref{operatorf}), is Frechét-differentiable, $\lambda\left(D\mathcal{F}(u_0)\right)\in\left[L-\frac{N}{\pi m},L+\frac{N}{\pi m}\right]$, where $L$ is defined as \[L=\min\left(\lambda\left(D\mathcal{F}(u_0)_{\cos, m}\right), \dfrac{(\pi m)^2}{\omega(m)}\right),\] and moreover, a lower bound for $\lambda(D\mathcal{F}(u_0)_{\cos,m})$ can be computed using interval arithmetic, see Remark \ref{computeBijMod}.

\vspace{0.3cm}
\noindent \textbf{Computation of $K$:}
\vspace{0.3cm}

Once again following \cite{RamosTwopoint2020}, if $r>0$ and
\begin{equation*}
\begin{aligned}
c_1\left\|f_{uu}(x,u)\right\|\leq K \mbox{ for all }(x,u)\in I'\times [-c_1r,c_1r],
\end{aligned}
\end{equation*}
then $K$ is a Lipschitz constant for $D\mathcal{F}$ in
 $\bar{B}\left(0,r\right)_{H^2(I)}$.
 
\vspace{0.3cm}
\noindent \textbf{Computation of the sequence $\{u_k\}_{k\in \N}$ and error terms $r_k$:}
\vspace{0.3cm}

 The sequence $\{u_k\}_{k\in \N}$ for the Inexact Newton Method for $\mathcal{F}$ in $H^2_N(I)$ can be chosen computing non-rigorously the sequence $\{b_k\}_{k\in \N}$ for the Newton Method $b_{k+1}=T_{\mathcal{F}_{\cos,m}}(b_k)$, $b_0=h_{\cos}(u_0)\in \R^m$ for the finite dimensional function $\mathcal{F}_{\cos,m}:\R^m\to \R^m$ and then letting $u_k=h^{-1}_{\cos,m}(b_k)$ for all $k\in \N$. Since all $u_k$ will be elementary functions, it follows that rigorous enclosures for the errors $\left\|r_k\right\|=\left\|u_{k+1}-T_{\mathcal{F}}(u_k)\right\|$ can be computed using the Simpson rule with rigorous error terms over the error formula given by item $(v)$ of Theorem \ref{fundtheobij}.

\begin{example}\label{example1} 
Let $\mathcal{F}:H^2_N(I)\to L^2(I)$ be a operator defined by $\mathcal{F}(u)(x) =u''(x) - \left(\sin(u(x)) - \cos(2\pi x)\right)$ for all $u\in H^2_N(I)$ and $x\in I$. From \cite{RamosTwopoint2020}, it is easy to see that $b_0=(0,0,0.7)$ is an approximate zero of $\mathcal{F}_{\cos,3}$ in $\bar{B}(0,1)_{\R^3}$ and thus we let
\begin{equation*}u_0=h_{\cos}^{-1}(b_0)=\frac{7}{10}\frac{\sqrt{2}\cos(2 \pi x)}{w(2)}
\end{equation*}
as an approximate zero for $\mathcal{F}$ in $\bar{B}(0,1)_{H^2_N(I)}$. Thus, following the above discussion, we compute through interval means the constants
\begin{equation*}K=c_1^2=\tanh(1)\mbox{ and }N = 2.01
\end{equation*}
such that $K$ is a Lipschitz constant for $\mathcal{F}$ in $\bar{B}(0,1)_{H^2_N(I)}$ and $\lambda(D\mathcal{F}(u_0))\in\left[L-\frac{N}{\pi m},L+\frac{N}{\pi m}\right]$, where
\begin{equation*}
\begin{aligned}
 \lambda(D\mathcal{F}(u_0)_{\cos,3})\geq  \left\|D\mathcal{F}(u_0)_{\cos,3}^{-1}\right\|_F^{-1} \geq 0.58\Rightarrow\\
 L= \min\{\lambda(D\mathcal{F}(u_0)_{\cos,3}),\frac{(3 \pi)^2}{w(3)}\}\geq 0.58
 \end{aligned}
 \end{equation*}
 and thus
 \begin{equation*} \lambda(D\mathcal{F}(u_0))\geq  0.58-\frac{2.01}{3\pi}\geq 0.36 = \nu.
\end{equation*}
Moreover, we computed directly from the Simpson rule with explicit error terms that
\begin{equation*}
\begin{aligned}
 \|\mathcal{F}(u_0) \|_{L^2(I)}\leq 2\cdot 10^{-3}=\eta.\\
\end{aligned}
\end{equation*}
Finally, letting $d= 10^{-2}$ we obtain a zero $t_d^*$ for $g_d(t)=\eta_d - \nu_d t + \frac{3K}{2}t^2$ satisfying $t_d^*\leq 2\cdot 10^{-2}$. Thus, all conditions of Theorem \ref{fundtheobij} are satisfied and we can apply the Inexact Newton Method. From item $(i)$ of Moreover, from theorem \ref{fundtheobij}, there exists a zero $u^*$ of $\mathcal{F}$ in $\bar{B}(0,1)_{H^2_N(I)}$ such that
\begin{equation*}\left\|u^*-u_0\right\|_{H^2(I)}\leq 2 \cdot 10^{-2}.
\end{equation*}

\begin{figure}[h!]\label{figm1}
\centering
 \includegraphics[width=\textwidth]{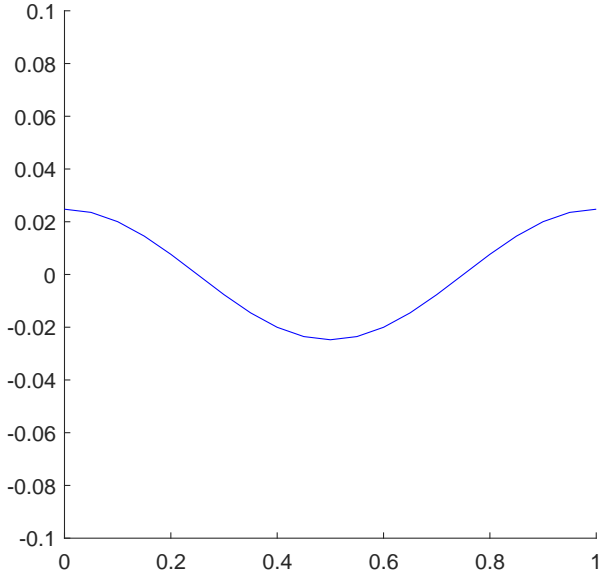}
\caption{Numerical zero $u_3$ of $\mathcal{F}$.}
\end{figure}

The next candidates $u_1$, $u_2$, $u_3$ for the Inexact Newton Method were computed using  by $u_i=h_{\cos}^{-1}(b_i)$ for $1\leq i\leq 3$, where $b_i$, $1\leq i\leq 3$ are the next terms in the Newton Method for $\mathcal{F}_{\cos,10}$ starting at $b_0=(0,0,0.7,\cdots,0)\in \R^{10}$. Thus, using item $(v)$ of Theorem \ref{fundtheobij} the errors $r_i=u_{i+1}-T_f(u_i)$, $0\leq i\leq 2$ were proven to satisfy
\begin{equation*}
\begin{aligned}
\left\|r_i\right\|\leq 2\cdot 10^{-10}\mbox{ for }0\leq i\leq 2.
\end{aligned}
\end{equation*}
Therefore, since $\dfrac{\sqrt{3}}{2}K\nu^{-1}\leq 3.16$ using recursively item $(iii)$ of Theorem \ref{fundtheobij} we conclude that
\begin{equation*}\left\|u^*-u_3\right\|_{H^2(I)}\leq 3\cdot 10^{-10}.
\end{equation*}
Thus, using Theorem \ref{fundtheobij} and the Inexact Newton Method, it was possible to obtain a better approximation for a zero of $\mathcal{F}$ compared to the one obtained in \cite{RamosTwopoint2020}.

\end{example}

\section{Conclusion}\label{sec:Conclusion}

In this paper, we proposed a new theorem for the feasibility and convergence of the inexact Newton method, with explicit convergence rate formulas similar to that of the Newton-Kantorovich theorem. After that, we connected our new definition of bijectivity modulus with this new theorem to verify with rigour zeros for a differential operator $\mathcal{F}$, using the inexact Newton method.

As we saw in Example \ref{example1}, the inexact Newton method together to bijectivity modulus, built a powerful tool to verify with rigour zeros for a non-linear differential operator. Moreover, note that the non-linearity was not a problem in this kind of approach, against the most part of methods. Furthermore, this method allow us to locate the true zero and find out better (faster and more accurate) numerical solutions compared to previous work of the authors of this paper.




%
%

\bibliographystyle{spmpsci}      
\bibliography{reference}   

%
%

\end{document}